\documentclass[a4paper, 12pt]{article}
\usepackage{amsmath, amsthm, amssymb}
\usepackage{type1cm}
\usepackage{mathptmx}
\usepackage[shortlabels]{enumitem}
\usepackage{graphicx}

\numberwithin{equation}{section}

\theoremstyle{definition}
\newtheorem{theorem}{Theorem}[section]
\newtheorem{lemma}[theorem]{Lemma}
\newtheorem{cor}[theorem]{Corollary}

\newtheorem{remark}[theorem]{Remark}
\newtheorem{definition}[theorem]{Definition}
\newtheorem*{claim}{Claim}
\newcommand{\es}{\mathcal{S}}
\newcommand{\esns}{\es^{\text{\/ nonsep}}}
\newcommand{\essep}{\es_{\text{sep}}}
\newcommand{\ins}{\iota^{\text{nonsep}}}
\newcommand{\isep}{\iota_{\text{sep}}}
\newcommand{\Z}{{\mathbb Z}}
\newcommand{\Aut}[1]{{\rm Aut}\left( #1 \right)}
\newcommand{\M}{\mathcal{M}}
\newcommand{\gi}{i_{\text{geom}}}
\newcommand{\jfi}{i_{\text{JF}}}

\DeclareMathOperator{\Ker}{Ker}

\begin{document}
\title{Geometric intersection in representations of mapping class groups of surfaces}
\author{Yasushi Kasahara}
\date{\today}
\bibliographystyle{amsplain}

\maketitle

\begin{abstract}
 We show that the detection of geometric intersection in an arbitrary representation
 of the mapping class group of surface implies the injectivity of that representation up to center,
 and vice versa.
 As an application, we discuss the geometric intersection in the Johnson filtration. Also, we
 further consider the problem of detecting the geometric intersection between separating
 simple closed curves in a representation.
\end{abstract}
\section{Introduction} \par

Detecting geometric intersection can be a powerful tool for the study of representations of the
mapping class group of surface. For instance, a certain kind of such detection
in the Lawrence--Krammer representation by Bigelow
\cite{bigelow:linearity} led to an affirmative solution to the linearity problem
for Artin's braid group which is nothing but the mapping class group of  a punctured disk. Conversely, 
the impossibility of detecting a similar kind of geometric intersection
had led to the unfaithfulness results for the Burau representation of the braid group 
as shown by Moody \cite{moody}, Long--Paton \cite{long-paton},
and Bigelow \cite{bigelow:burau}.
As for the mapping class group of a surface of  higher genus,
this type of result was given by Suzuki \cite{suzuki_2} for
the Magnus representation of the Torelli group.
In each of all these works, it was fundamental to establish a criterion that the representation
in question can detect the geometric intersection if and/or only if its kernel is small.
\par

In this paper, instead of considering any particular representation, 
we derive a similar criterion applicable to an {\em arbitrary group homomorphism} of the mapping class group
of a surface of genus at least one,
by focusing our attention on the following fact:
\begin{quote}
 The geometric intersection number between two simple closed curves is zero
 if and only if the commutator of the two Dehn twists along them represents
 the identity in the mapping class group.
\end{quote}
\par\goodbreak\medskip
We now describe our main result.
Let $\Sigma_{g,n}$ be an oriented compact connected surface of
genus $g \geq 1$ with $n \geq 0$ boundary components.
The mapping class group $\M_{g,n}$ of $\Sigma_{g,n}$ is defined as the group of all the isotopy classes of
orientation preserving homeomorphisms of $\Sigma_{g,n}$ where all homeomorphisms and isotopies
are assumed to preserve the boundary of $\Sigma_{g,n}$ pointwise. Let $\es$ be the set of all the isotopy
classes of essential simple closed curves on $\Sigma_{g,n}$. Here,  essential is meant to be  not
homotopic to a point nor parallel to any of the boundary components.
For $c \in \es$, we denote by $t_c$ the right-handed Dehn twist along $c$.
We denote by $\esns$ the subset of $\es$ consisting of all the isotopy classes of nonseparating
simple closed curves. The {\em commutator} of two elements $x$ and $y$ in a group is defined by
$[x, y] = x y x^{-1} y^{-1}$.
\par

Our criterion states that 
the triviality of geometric intersection number for {\em all} pairs of essential
simple closed curves can be detected by a homomorphisms of $\M_{g,n}$ if and only if
its kernel is small:

\begin{theorem}\label{thm:main}
 Let $G$ be a group and  $\rho: \M_{g,n} \to G$ an arbitrary homomorphism. If
 $[t_{c_1}, t_{c_2}] =1$ for those $c_1$, $c_2 \in \esns$ which satisfy $\rho( [t_{c_1}, t_{c_2}] ) = 1$,
 then the kernel of $\rho$ is contained in the center $Z(\M_{g,n})$ of $\M_{g,n}$.
 Conversely, if $\Ker{\rho} \subset Z(\M_{g,n})$, then $[t_{c_1}, t_{c_2}] = 1$ for any
 $c_1$ and $c_2 \in \es$ which satisfy $\rho([t_{c_1}, t_{c_2}]) = 1$.
\end{theorem}
\par
Note that the curves $c_1$ and $c_2$ 
need not be nonseparating in the latter half of Theorem \ref{thm:main}.

\begin{remark}
 The structure of the center $Z(\M_{g,n})$ is well-known due to Paris--Rolfsen \cite{pr}.
 If  $n=0$, $Z(\M_{g,n})$ is trivial except for the case $g \leq 2$, where the center
 is generated by the class of hyperelliptic involution. For the case of $g=1$ and $n=1$,
 the center is an infinite cyclic group generated by the ``half-twist'' along the unique
 boundary component. For all the other cases, $Z(\M_{g,n})$ is a free abelian group of
 rank $n$ and is generated by the Dehn twists along the boundary components of $\Sigma_{g,n}$.
\end{remark}
\par\goodbreak\medskip
The organization of this paper is as follows.
The proof of Theorem \ref{thm:main} is given in Section \ref{proof:criterion}
after necessary preparation. In Section \ref{johnson_filtration}, as an application, 
we discuss the geometric intersection in the Johnson filtration and pose a certain problem.
Also, in Section \ref{separating_curves}, we further consider the geometric intersection between
{\em separating} simple closed curves and provide a criterion similar to 
Theorem \ref{thm:main}. In Appendix \ref{appendix}, we give a proof of certain 
key lemma for Section \ref{separating_curves}.

\par

\section{Proof of Theorem \ref{thm:main}} \label{proof:criterion} \par

We first prepare some necessary results. We refer to \cite{fm} as basic reference for mapping class
groups of surfaces. We also need some results
in our previous work \cite{kasahara:GD} with certain modification.
\par

For $a$, $b \in \es$, the geometric intersection number, denoted by $\gi(a,b)$, is the minimum of the
number of the intersection points of the simple closed curves $\alpha$ and $\beta$
where $\alpha$ and $\beta$ vary the isotopy classes of $a$ and $b$, respectively. It defines a function
   $$\gi: \es \times \es \to \Z_{\geq 0}.$$
The following is the precise statement of the fact mentioned in Introduction (c.f.~Fact 3.9 in \cite{fm}).

 \begin{lemma}
  \label{fact:1}
 For $c_1$, $c_2 \in \es$, $\gi(c_1, c_2) = 0$ if and only if $[t_{c_1}, t_{c_2}] = 1$ in $\M_{g,n}$.
 \end{lemma}
 \par\goodbreak\medskip

 The next is also well-known, and will be crucial in our argument.
 \par

 \begin{lemma} \label{fact:2}
 Suppose $c_1$, $c_2 \in \es$. If $c_1 \neq c_2$, then there exists $d \in \es$ such that
 $\gi(c_1, d) = 0$ and $\gi(c_2, d) \neq 0$.
 \end{lemma}
The proof can be found in \cite[page 73]{fm}.

We define the mapping
$\iota : \es \to \M_{g,n}$ by $\iota(c) := t_c$ for $c \in \es$.
It is known that $\iota$ is injective (c.f. \cite[Fact 3.6]{fm}).
As explained in \cite[Lemma 3.2]{kasahara:GD},
the proof of this fact, which depends on Lemma \ref{fact:2}, actually implies the following.

\begin{lemma}\label{fact:2.5}
 For any $c_1$ and $c_2 \in \es$, the element
 $\iota(c_1) \iota(c_2)^{-1}$ lies in the center $Z(\M_{g,n})$  if and only if $c_1 = c_2$.
\end{lemma}
The following is a slight generalization of our previous result \cite{kasahara:GD}, where
we dealt with $\es$ instead of $\esns$. We denote by $\ins$ the restriction of $\iota$ to $\esns$.
\par

\begin{lemma} \label{fact:3}
 Let  $G$ be a group and $\rho : \M_{g,n} \to G$ an arbitrary homomorphism. 
 \begin{enumerate}[label={\normalfont (\arabic*)},topsep=2pt,itemsep=1pt,parsep=1pt]
  \item If the mapping $\rho \circ \ins$ is injective, then $\Ker{\rho} \subset Z(\M_{g,n})$.
  \item If $\Ker{\rho} \subset Z(\M_{g,n})$, then the mapping $\rho \circ \iota$ is injective.
 \end{enumerate}
\end{lemma}
\par
The second part of Lemma \ref{fact:3} follows from \cite[Lemma 3.2]{kasahara:GD}. 
The subtle point for the first part is to observe that one can choose a generating set for $\M_{g,n}$
due to Gervais \cite{gervais} from
the image of $\ins$ and the Dehn twists along boundary components of $\Sigma_{g,n}$. This follows
actually from the construction there.
\par
 \begin{proof}[Proof of Lemma \ref{fact:3}]It suffices to prove the first part.
  We first recall the effect of the natural action of $\M_{g,n}$ on $\es$ over the image of $\iota$.
  For any $f \in \M_{g,n}$ and $c \in \es$, it holds
  \begin{equation} \label{effect}
   \iota(f(c)) = f \cdot \iota(c) \cdot f^{-1}
  \end{equation}
  (c.f. \cite[Fact 3.7]{fm}).
  \par
  Suppose next that the mapping $\rho \circ \ins$ is injective. Let $f \in \Ker{\rho}$.
  Then for any $c \in \esns$, we have
  $$\rho \circ \ins (f(c)) = \rho(f \cdot \ins(c) \cdot f^{-1})  = \rho \circ \ins(c).$$
  Hence we have $f(c) = c$ for each $c \in \esns$. In view of \eqref{effect}, this shows that
  $f$ commutes with each element of $\ins(\esns)$.   As mentioned above, the mapping class
  group $\M_{g,n}$ is generated by $\ins(\esns)$ together with the Dehn twists along boundary
  components of $\Sigma_{g,n}$. The latter type of mapping classes obviously lie in the center $Z(\M_{g,n})$.
  Therefore, we have $f \in Z(\M_{g,n})$, and hence $\Ker{\rho} \subset Z(\M_{g,n})$.
  This completes the proof of Lemma \ref{fact:3}.  
 \end{proof}
\par\goodbreak\medskip

We are now ready to prove Theorem \ref{thm:main}. Let $\rho: \M_{g,n} \to G$ be an
arbitrary group homomorphism. Suppose $[t_{c_1}, t_{c_2}] = 1$ for those $c_1$, $c_2 \in \esns$
which satisfy $\rho([t_{c_1}, t_{c_2}]) = 1$.
For the first part of the theorem, by virtue of Lemma \ref{fact:3} (1),
it suffices to show that $\rho \circ \ins$ is injective.
For any $c_1$, $c_2 \in \esns$ with $c_1 \neq c_2$, by Lemma \ref{fact:2}, we may choose
$d \in \es$ such that $\gi(c_1, d) = 0$ and $\gi(c_2,d) \neq 0$.
Furthermore, since $c_1$ and $c_2$ are nonseparating, it can
be seen that we may assume $d$ is also nonseparating.
Then by Lemma \ref{fact:1} and the assumption of the theorem, respectively, we have
$$\rho( [t_{c_1}, t_d] ) = 1, \quad \text{and} \quad \rho( [t_{c_2}, t_d]) \neq 1.$$
This implies $\rho \circ \ins(c_1) \neq \rho \circ \ins(c_2)$,
which shows that $\rho \circ \ins$ is injective.
\par
Next, suppose conversely that $\Ker{\rho} \subset Z(\M_{g,n})$.
Then $\rho \circ \iota$ is injective by Lemma \ref{fact:3} (2).
For $c_1$, $c_2 \in \es$, assume $\rho([t_{c_1}, t_{c_2}]) = 1$. Since
$[t_{c_1}, t_{c_2}] = \iota(t_{c_1}(c_2)) \cdot \iota(c_2)^{-1}$ in view of \eqref{effect},
we have $\iota(t_{c_1}(c_2)) \cdot \iota(c_2)^{-1} \in Z(\M_{g,n})$. We can then conclude
$t_{c_1}(c_2) = c_2$ by Lemma \ref{fact:2.5}. This implies $[t_{c_1}, t_{c_2}] = 1$.
This completes the proof of Theorem \ref{thm:main}.
\qed
\par

\section{The Johnson filtration} \label{johnson_filtration} \par

Let $g \geq 2$. We now consider the relation of geometric intersection with the Johnson filtration.
We consider only the case of $n=1$ in order to avoid the problem that there are no canonical
choices of the filtration for $n >1$. 
\par

Let $\Gamma$ be the fundamental group of the
surface $\Sigma_{g,1}$ with a fixed base point on the boundary,
which is a free group of rank $2g$. The lower central series 
of $\Gamma$, denoted by $\{ \Gamma_k \}_{k \geq 1}$, is defined recursively by
$\Gamma_1 = \Gamma$, and $\Gamma_k = [\Gamma, \Gamma_{k-1}]$ for $k \geq 2$. For each $k \geq 1$,
$\Gamma_k$ is a characteristic subgroup of $\Gamma$, so that the natural action of $\M_{g,1}$ on $\Gamma$
gives rise to the one on the quotient nilpotent group $N_k := \Gamma/\Gamma_{k+1}$. The latter action induces
a homomorphism, which we denote by
   $$\rho_k: \M_{g,1} \to \Aut{N_k}.$$
We denote the kernel of $\rho_k$ by $\M(k)$. These $\M(k)$'s form a descending central filtration
of $\M_{g,1}$ which is called the Johnson filtration. It follows, by definition,  that $\M(k)$
is the domain of the $k$th Johnson homomorphism for $k \geq 1$, and is also the kernel
of the $k-1$st Johnson homomorphism for $k \geq 2$.
\par\goodbreak\medskip

The following shows, in view of Lemma \ref{fact:1},
that the Johnson filtration in any finite depth does not detect
the geometric intersection.

\begin{cor}\label{cor}
 For any integer $k \geq 1$, there exists a pair of simple closed curves
 $c_1$ and $c_2 \in \es$ such that the commutator of the Dehn twists along them lies in $\M(k)$
 but is not the identity. Furthermore, one can always choose such $c_1$ and $c_2$ from
 $\esns$.
\end{cor}
\par

\begin{proof}
 Suppose to the contrary that there exists some $k_0 \geq 1$ such that the condition
 $[t_{c_1}, t_{c_2}] \in \M(k_0)$ for any $c_1$, $c_2 \in \esns$
 implies $[t_{c_1}, t_{c_2}] = 1$ in $\M_{g,1}$. 
 Then we can apply Theorem \ref{thm:main} for $\rho = \rho_{k_0}$ to see that
 $\M(k_0)$ must be contained in the center $Z(\M_{g,1})$.
 Therefore, it is sufficient to confirm:

\begin{claim} \label{claim:noncentral}
 For each $k \geq 1$, $\M(k)$ is not contained in the center $Z(\M_{g,1})$.
\end{claim}

This is well-known, but we provide a short proof for completeness.
 Let $a$, $b$ be two separating essential  simple  closed curves with $\gi(a, b) \geq 2$.
 Then the two Dehn twists along $a$ and $b$ generate
 a free group $F$ of rank $2$, due to the work by Ishida \cite{ishida}.
 Take the lower central series of $F$, each term of
 which is obviously non-trivial and is not contained in $Z(\M_{g,1})$.
 On the other hand, by the work of Johnson \cite{johnson:abelian},
 the Dehn twist along any {\em separating} essential simple closed curve lies in $\M(2)$, and therefore
 $F$ is contained in $\M(1)$. By Morita \cite{morita:duke}, we have $[\M(k), \M(l)] \subset \M(k+l)$.
 Hence for each $k \geq 1$, the $k$-th term of the lower central series of $F$ is contained in $\M(k)$.
 This proves the claim, and hence Corollary \ref{cor}.
\end{proof}
\par\goodbreak\medskip

\subsection*{Towards refinement} \par
In spite of Corollary \ref{cor}, the totality of the Johnson filtration can
detect the geometric intersection, due to the following:
\begin{theorem}[Johnson \cite{johnson:survey}] \label{fact:4}
 $$\bigcap_{k \geq 1}{\M(k)} = \{ 1 \}.$$ 
\end{theorem}

In view of this fact, the following definition would be natural:

\begin{definition}
 For $c_1$ and $c_2 \in \es$, we define
 $$\jfi(c_1, c_2) :=
  \begin{cases}
   1 \quad \text{if $[t_{c_1}, t_{c_2}] \notin \M(1)$;} \\
   k+1 \quad \text{if $[t_{c_1}, t_{c_2}] \in \M(k)$ and $[t_{c_1}, t_{c_2}] \notin \M(k+1)$;} \\
   0 \quad \text{if $[t_{c_1}, t_{c_2}] \in \M(k)$ for all $k \geq 1$.}
  \end{cases}$$
\end{definition}

In this terminology, Corollary \ref{cor} can be rephrased as saying that the function
$\jfi$ is unbounded. This function should measure a kind of  complexity of
the configuration of two simple closed curves, and possibly has some relation with the geometric intersection
number. However, we do not know any explicit relationship except for two cases: 
$\jfi(c_1, c_2) = 0$ if and only if $\gi(c_1, c_2) = 0$, which is  due to Theorem \ref{fact:4} and
Lemma \ref{fact:1}; $\jfi(c_1,c_2) \geq 2$ if and only if $\gi(c_1, c_2) \neq 0$ and
the algebraic intersection number
of $c_1$ and $c_2$ with respect to arbitrarily fixed orientations on them is zero.
We also note that the latter shows $\gi(c_1, c_2) = 1$ implies 
$\jfi(c_1, c_2) = 1$. 
It would be an interesting problem to study further relation.
\par

\section{The case of separating curves}\label{separating_curves} \par

Let $g \geq 1$ and $n \geq 0$ again.
We denote by  $\essep$ the complement of $\esns$ in $\es$, which consists of the isotopy 
classes of essential {\em separating} simple closed curves on $\Sigma_{g,n}$. 
It is easy to see that $\essep$ is empty if and only if the Euler number 
$\chi(\Sigma_{g,n}) \geq -1$, {\em i.e.}, $g=1$ and $n \leq 1$. 
We can prove that Theorem \ref{thm:main} still holds true if $\esns$ is replaced by 
$\essep$, so far as $\essep$ is not empty.
\par\goodbreak\medskip

We assume in this section $\chi(\Sigma_{g,n}) \leq -2$ so that $\essep$ is not empty. 
Let $\isep \colon \essep \to \M_{g,n}$ denote the restriction of $\iota$ to $\essep$.
\begin{theorem} \label{thm:B}
 Assume $\chi(\Sigma_{g,n}) \leq -2$. Let $G$ be a group, and $H$ an arbitrary subgroup of $\M_{g,n}$
 which contains the image of $\isep$. Suppose $\rho \colon H \to G$ is an arbitrary homomorphism.
 Then the following holds.
 \begin{enumerate}[label={\normalfont (\arabic*)},topsep=2pt,itemsep=1pt,parsep=1pt]
  \item If $[t_{c_1}, t_{c_2}] = 1$ for those $c_1$, $c_2 \in \essep$ which satisfy
	$\rho([t_{c_1}, t_{c_2}]) = 1$, then $\Ker{\rho} \subset Z(\M_{g,n})$.
  \item If $\Ker{\rho} \subset Z(\M_{g,n})$, then $[t_{c_1}, t_{c_2}] = 1$ for
	those $c_1$, $c_2 \in \essep$ which satisfy $\rho([t_{c_1}, t_{c_2}]) = 1$.
 \end{enumerate}
\end{theorem}

\begin{remark}
 Theorem \ref{thm:B} seems to shed some light on the significance of the work by Suzuki. He gave in
 \cite{suzuki} explicit elements, as commutators of two Dehn twists along separating essential simple
 closed curves, of the kernel of the representation mentioned in Introduction, the
 linear representation of the Torelli group $\M(1)$ for
 $n=1$ which is defined as the Magnus representation associated with the abelianization of $\Gamma$.
\end{remark}

The proof of Theorem \ref{thm:B} is essentially the same as that of Theorem \ref{thm:main},
and here we give just a sketch of it, except for Lemmas \ref{fact:5} and \ref{lemma:A} below.
First, a little care in proving Lemma \ref{fact:2} obtains:

\begin{lemma}\label{fact:2prime}
 Suppose $c_1$, $c_2 \in \essep$. If $c_1 \neq c_2$, then there exists $d \in \essep$
 such that $\gi(c_1, d)= 0$ and $\gi(c_2, d) \neq 0$.
\end{lemma}

The following is a consequence of the results of Brendle--Margalit \cite{bm} and Kida \cite{kida} with a finite number of exceptions on $(g, n)$:

\begin{lemma} \label{fact:5}
	Assume $\chi(\Sigma_{g,n}) \leq -2$. A mapping class of $\M_{g,n}$ acts 
	trivially on $\essep$ if and only if it lies in the center $Z(\M_{g,n})$.
\end{lemma}

We include a proof of Lemma \ref{fact:5} as Appendix \ref{appendix} to confirm there are no exceptional cases.
\par\goodbreak\medskip

By making use of Lemmas \ref{fact:5} and \ref{fact:2.5}, we have the following analogue of Lemma \ref{fact:3}:

\begin{lemma} \label{lemma:A} Suppose $\chi(\Sigma_{g,n}) \leq -2$.
 Let $G$ be a group, and $H$ an arbitrary subgroup of $\M_{g,n}$ which contains the image of $\isep$.
 For any homomorphism $\rho: H \to G$, the mapping $\rho \circ \isep$ is injective
 if and only if $\Ker{\rho} \subset Z(\M_{g,n})$.
\end{lemma}

\begin{proof}
	Suppose that $\rho \circ \isep$ is injective. Then, in view of \eqref{effect}, any element of $\Ker{\rho}$ acts trivially on $\essep$, as in the proof of Lemma \ref{fact:3}. This implies $\Ker{\rho} \subset Z(\M_{g,n})$ by Lemma \ref{fact:5}. Conversely, suppose $\Ker{\rho} \subset Z(\M_{g,n})$. Then for any $c_1$, $c_2 \in \essep$ with $\rho \circ \isep(c_1) = \rho \circ \isep(c_2)$, we have $\rho( \iota(c_1) \iota(c_2)^{-1} ) = 1$ and hence $\iota(c_1) \iota(c_2)^{-1} \in Z(\M_{g,n})$. We then have $c_1 = c_2$ by Lemma \ref{fact:2.5}. This shows $\rho \circ \isep$ is injective.
\end{proof}

Now by appealing to Lemma \ref{fact:2prime}, the same argument for Theorem \ref{thm:main}  completes the proof of Theorem \ref{thm:B}.
\qed
\par\goodbreak\medskip

\begin{remark}
	Even if $g=0$, all the results in this section hold true under the same 
	assumption $\chi(\Sigma_{g,n}) \leq -2$. The proof is also the same, except
	that $\M_{g,n}$ is generated by $\isep(\essep)$, rather than $\ins(\esns)$, 
	together with $Z(\M_{g,n})$.
\end{remark}
\par

\appendix
\section{Proof of Lemma \ref{fact:5}} \label{appendix}\par

Crucial is Lemma \ref{lemma:C}, which is proved following the idea
of ``sharing pair" introduced by Brendle--Margalit \cite{bm}.
We note in case $g=0$ that Lemma \ref{fact:5} holds true, and the proof 
is easier since $\essep = \es$.

\par\goodbreak\medskip

Hereafter, we assume $g \geq 1$, $n \geq 0$, and $\chi(\Sigma_{g,n}) \leq -2$.
\par\goodbreak\medskip

We first recall the mapping $\iota$ is injective. Therefore, we see that 
any mapping class in the center $Z(\M_{g,n})$ acts trivially on $\essep$,
due to \eqref{effect}.
\par

Next,  we show that any mapping class which acts trivially on $\essep$ lies 
in $Z(\M_{g,n})$. This follows from the next lemma.

\begin{lemma}\label{lemma:C}
	Any mapping class of $\M_{g,n}$ 
	which acts trivially on $\essep$ acts also trivially on $\es$.
\end{lemma}

Indeed, any mapping class which acts trivially on $\es$ commutes with all the 
Dehn twists, due to \eqref{effect}. As we noted before, $\M_{g,n}$ is 
generated by $\iota(\es)$ together with $Z(\M_{g,n})$. Therefore, 
such a mapping class lies in $Z(\M_{g,n})$. Hence, Lemma \ref{lemma:C} 
implies the desired result.

\begin{proof}[Proof of Lemma \ref{lemma:C}]
	Suppose $f \in \M_{g,n}$ acts trivially on $\essep$. We need to show 
	$f$ acts also trivially on $\esns$. Since $\chi(\Sigma_{g,n}) \leq -2$, we may choose
	a subsurface $S$ homeomorphic to $\Sigma_{1,2}$ in the interior of $\Sigma_{g,n}$ 
	so that its two boundary components are either
	essential simple closed curves on $\Sigma_{g,n}$, or parallel to boundary components of
        $\Sigma_{g,n}$.
	Unless $(g,n)=(2,0)$, we may further assume the 
	two boundary components of $S$ are 
	not isotopic in $\Sigma_{g,n}$.
	We take simple closed curves $A$, $B$, and $C$ on $\Sigma_{g,n}$ as depicted in Figure \ref{C-1} (a) and (b)
	where only a neighborhood of $S$ is drawn. 
	\par
		\begin{figure}[!h]
			\begin{center}
				\includegraphics[width=0.8\linewidth]{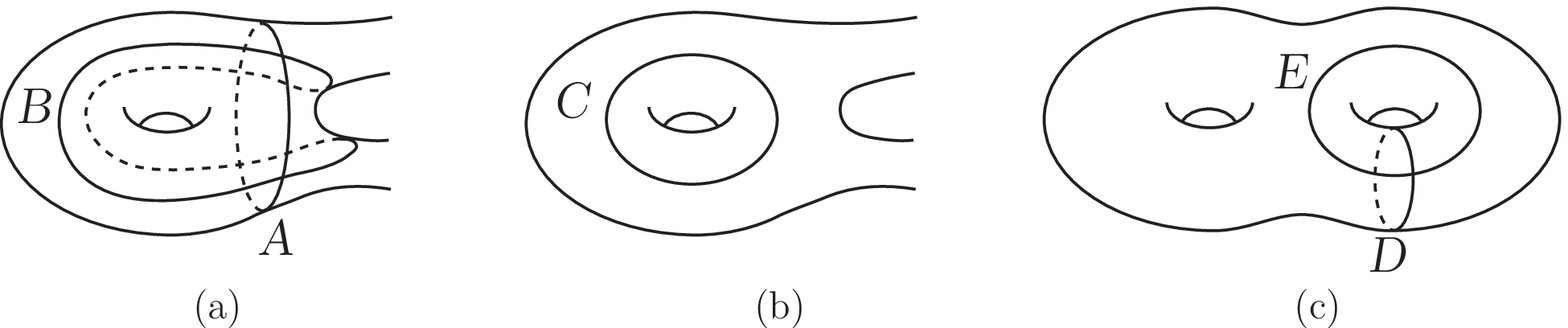}
			\end{center}
			\caption{}\label{C-1}
		\end{figure}
	\par
	We denote the isotopy classes of them by $a$, $b$, and $c \in \es$, respectively.
	Consider the pair $(A, B)$\footnote{This pair forms a sharing pair of Brendle--Margalit if the complement of the union of the two subsurfaces of $S$ both homeomorphic to $\Sigma_{1,1}$ bounded by $A$ and $B$, respectively 
		is connected.}. 
	Since $A$ and $B$ are essential and separating in $\Sigma_{g,n}$, we have $f(a) = a$ and $f(b) = b$.
	We also see $A$ and $B$ are in minimal position, and $a \neq b$. Then, we
	can see there exists a homeomorphism $F: \Sigma_{g,n} \to \Sigma_{g,n}$
	such that $F$ represents $f \in \M_{g,n}$ and preserves $A \cup B$ 
	setwise (see Lemma 2.9 in \cite{fm}.) In particular, $F$ induces a 
	homeomorphism of the complement of $A \cup B$. Also, the restriction of
	$F$ to the boundary is the identity. On the other hand, the simple closed
	curve $C$ is the core of an annulus component of 
	$\Sigma_{g,n} \smallsetminus (A\cup B)$
	which is disjoint from the boundary of $\Sigma_{g,n}$. 
	\par
	
	Unless $(g,n) = (2,0)$, such a component is unique, and we have $f(c) = c$.
	\par
	
	If $(g,n) = (2,0)$, then the above argument only shows either $f(c) = c$ or 
	$f(c) = d$ where $d$ denotes the isotopy class of the 
	simple closed curve $D$ depicted in Figure \ref{C-1} (c).
	Now we take $e \in \es$ represented by the simple closed curve $E$ 
	depicted in Figure \ref{C-1} (c), and consider another pair $(T_e(A), T_e(B))$
	where $T_e$ denotes any representative homeomorphism for $t_e$. Then 
	by the same argument we have either $f(c) = c$ or $f(c) = t_e(d)$.
	Since $t_e(d) \neq d$, we have $f(c) = c$. 
	\par
	Consequently, we have $f(c) = c$ for the particular $c \in \esns$ 
	in any case.
	Since any two elements of $\esns$ are mapped to each other by some 
	mapping class, the argument above shows $f$ acts trivially on $\esns$. 
	This completes the proof of Lemma \ref{lemma:C}.
\end{proof}

\subsubsection*{Acknowledgements}\par
The work in this paper started as an attempt to determine whether or not the function $\jfi$
introduced in Section \ref{johnson_filtration} is bounded.
The author is grateful to Ryosuke Yamamoto for drawing the author's attention
to the problem through his talk in the conference
``Hurwitz action $\sim$HINERU$\sim$'' in January 2015. 
\par

The author is grateful to  Susumu Hirose for useful discussion concerning Appendix.
He is also grateful to Nariya Kawazumi, Shigeyuki Morita, and Andrew Putman for 
comments on an earlier version of this paper. 
Finally, he is grateful to the referee for helpful suggestions and corrections.
\par
The author was partially supported by JSPS KAKENHI Grant Number 16K05154.

\providecommand{\bysame}{\leavevmode\hbox to3em{\hrulefill}\thinspace}
\providecommand{\MR}{\relax\ifhmode\unskip\space\fi MR }
\providecommand{\MRhref}[2]{%
  \href{http://www.ams.org/mathscinet-getitem?mr=#1}{#2}
}
\providecommand{\href}[2]{#2}

\vspace{12pt}
\noindent
\textsc{Yasushi Kasahara \\
Department of Mathematics \\
  Kochi University of Technology \\ Tosayamada,  Kami City, Kochi \\ 
  782-8502 Japan} \\
E-mail: {\tt kasahara.yasushi@kochi-tech.ac.jp}

\end{document}